\documentclass[11pt,a4paper]{amsart}
\usepackage{a4wide,graphicx,xcolor}
\usepackage{amsmath,amssymb,array,amsthm,amsfonts}
\usepackage{mathtools}
\mathtoolsset{showonlyrefs}
\usepackage[T1]{fontenc}
\usepackage{}
\usepackage{mathabx}

\usepackage{enumitem}
\usepackage[utf8]{inputenc} 
\usepackage{mathrsfs}
\usepackage[toc,page]{appendix}

\usepackage[textwidth=15mm]{todonotes}

\setlist[itemize]{leftmargin=6mm}

\newcommand{\C}{\mathbb C}
\newcommand{\Pj}{\mathbb{P}}

\newcommand{\R}{\mathbb R}
\newcommand{\Q}{\mathbb Q}
\newcommand{\Z}{\mathbb Z}
\DeclareMathOperator{\codim}{codim}

\DeclareMathOperator{\Cl}{Cl}

\newcommand{\cO}{\mathcal{O}}

\DeclareMathOperator{\rk}{rk}

\DeclareMathOperator{\Hom}{Hom}

\DeclareMathOperator{\Pic}{Pic}

\newcommand\iso{\cong}

\newcommand{\N}{\mathbb{N}}

\newcommand{\res}{\operatorname{res}}

\newtheorem{thm}{Theorem}[section]

\newtheorem{prop}[thm]{Proposition}
\newtheorem{lma}[thm]{Lemma}
\newtheorem{cor}[thm]{Corollary}
\newtheorem{dfn}[thm]{Definition}

\theoremstyle{remark}
\newtheorem{rmkk}[thm]{Remark}
\newtheorem{exe}[thm]{Example}
\newenvironment{rmk}{\begin{rmkk}\rm}{\qee\end{rmkk}}
\newenvironment{ex}{\begin{exe}\rm}{\qee\end{exe}}

\newcommand{\qee}{\mbox{\hspace{0.2mm}}\hfill$\triangle$}



\title[ Noether-Lefschetz components in toric varieties]{An asymptotic description of the Noether-Lefschetz  \\[3pt] components in toric varieties}

\author{Ugo Bruzzo}
\address{Ugo Bruzzo, SISSA  (International School for Advanced Studies), Via Bonomea 265, 34136 Trieste, Italy; INFN (Istituto Nazionale di Fisica Nucleare), Sezione di Trieste, Trieste, Italy; IGAP (Institute for Geometry and Physics), Trieste, Italy }
\email{bruzzo@sissa.it}

\author{William D. Montoya}
\address{William D. Montoya, Instituto de Matemática, Estatística e Computação Científica, Universidade Estadual de Campinas (UNICAMP), Rua Sérgio Buarque de Holanda 651, 13083-859, Campinas, SP, Brazil}
\email{wmontoya@ime.unicamp.br}

\begin{document}

\begin{abstract}We extend the definition of Noether-Leschetz components 
 to quasi-smooth hypersurfaces in an odd dimensional projective  simplicial toric variety $\Pj_{\Sigma}^{2k+1},$ and   prove that asymptotically  the components whose codimension is upper bounded   by a suitable effective constant correspond to hypersurfaces with which one can associate a small $\phi$-degree $k$-dimensional subvariety; the  $\phi$-degree is a generalization of the notion of degree via a group homomorphism $\phi\in \Hom_{\Z}^+(\Cl(\Pj^{2k+1}_{\Sigma}),\Q)$ who plays an important role in  this work.  We conjecture and prove in some cases that this subvariety is contained in the hypersurface. As a corollary we get an asymptotic characterization  of the components with small codimension, generalizing  Otwinowska's work for $\Pj_{\Sigma}^{2k+1}=\Pj^{2k+1}$ and Green and Voisin's for $\Pj_{\Sigma}^{2k+1}=\Pj^3$. Some tools     developed in this paper are a generalization of Green's theorem for simplicial toric varieties, and an extension of the notion of artinian Gorenstein ideal for the  Cox ring of a toric variety.
  \end{abstract}

\maketitle

\setcounter{tocdepth}{1}
\tableofcontents

\section{Introduction}

The classical Noether-Lefschetz theorem states that a very general surface $X$ in $\Pj^3$ of degree $d\geq 4$ has Picard number $1$. In recent years generalizations have been proved using Hodge theory for simplicial projective toric threefolds  satisfying an explicit combinatorial condition   \cite{BruzzoGrassi2012},  and more generally by  Ravindra and Srinivas for normal projective threefolds using a purely algebraic approach  \cite{RavindraSrinivas2009} (see also \cite{BruGraLop}).

The Noether-Lefschetz locus is the subscheme of the (hyper)surface parameter space where the Picard number is greater than the Picard number of the ambient variety. Green  and Voisin  proved that if $NL_d$ is the Noether-Lefschetz  locus
for degree $d$ surfaces in $\Pj^3$,  with $d\ge 4$, the codimension of  every component of $NL_d$ is bounded from below by $d-3$, with equality exactly for the components formed by surfaces containing a line.  Otwinowska gave an asymptotic generalization of Green and Voisin's results to hypersurfaces in an odd dimensional projective space $\Pj^{2k+1}$  \cite{Otwinowska2003}.

{In \cite{BruzzoGrassi2018} (see also \cite{LanzaMartino}) it was  proved  that for simplicial projective toric threefolds the codimension of the Noether-Lefschetz components is also bounded from below.
There it was also proved that components corresponding to     surfaces containing a ``line,''  defined as a  curve which is minimal in a suitable sense, realize the lower bound. However  
the question whether these are exactly the components of smallest condimension  was left open. }

 The purpose of the present  paper is to extend and generalize Otwinowska's ideas to odd dimensional simplicial projective toric varieties.
 In section 2 we present   generalizations of Green's restriction theorem    and of a theorem of Macaulay
\cite{Green89},  while in section 3 we introduce a generalization of the notion of Gorenstein ideal, which we call a Cox-Gorestein ideal; these will be the key tools in the proof of our main result.
Section 4  is   more technical; there we introduce the notion of $\phi$-degree and prove some applications of the toric Green theorem to Cox-Gorenstein ideals.
In section 5 using Hodge theory
we explicitly construct  the tangent space at   points in the Noether-Lefschetz loci, which turns out to be a graded part of a Cox-Gorenstein ideal. In section 6 using all the machinery so far developed we prove our main result.

We shall consider a   projective simplicial toric variety $\Pj^{2k+1}_{\Sigma}$  {of dimension $2k+1$},  whose
fan is $\Sigma$, and   an ample line bundle  $L$ on $\Pj^{2k+1}_{\Sigma}$, 
with $\deg L= \beta \in \Pic(\Pj^{2k+1}_{\Sigma})$ satisfying $\phi(\beta)=n$  for some $n\ge 0$, where   $\phi\in \Hom_{\Z}^+(\operatorname{Cl}(\Pj^{2k+1}_{\Sigma}),\Q)$ 
	 and  $\eta$ is a class with $\phi(\eta)=1$; here $\Hom_{\Z}^+(\operatorname{Cl}(\Pj^{2k+1}_{\Sigma}),\Q)$ is the cone
	 of elements in $\Hom_{\Z}(\operatorname{Cl}(\Pj^{2k+1}_{\Sigma}),\Q)$ that are nonnegative on the effective classes in
	 $\operatorname{Cl}(\Pj^{2k+1}_{\Sigma})$. 
     We will consider a section
     $f\in  H^0(\cO_{\Pj^{2k+1}_{\Sigma}}(\beta)) $   such that $X_f=\{f=0\}$ is a quasi-smooth hypersurface.\footnote{Heuristically this means that $X_f$ has only singularities inherited from the ambient space, or more precisely, regarding $\Pj^{2k+1}_{\Sigma}$ as a smooth orbifold,
that $X_f$ is a smooth suborbifold, see e.g.~\cite{BatyrevCox}.}
Let $\mathcal{U}_{\beta}\subset \Pj(H^0(\cO_{\Pj^{2k+1}_{\Sigma}}(\beta))) $ be the open subset parameterizing the quasi-smooth hypersurfaces and let $\pi: \mathcal{X}_{\beta}\to  \mathcal{U}_{\beta}$ be its tautological family.
Let $H^{2k}_{\Q}$ be the local system $R^{2k}\pi_{\star}\Q$  and let $\mathcal{H}^{2k}$ be the locally free sheaf $H^{2k}_{\Q}\otimes \cO_{\mathcal{U}_{\beta}}$ over $\mathcal{U}_{\beta}$.

Let $0\neq \lambda_f\in H^{k,k}(X_f,\Q)/i^*(H^{k,k}(\Pj^{2k+1}_{\Sigma},\Q)$ and let $U$ be a contractible open subset around $f$, so that $\mathcal{H}^{2k}(U)$ is constant.
Finally, Let $\lambda\in \mathcal{H}^{2k}(U)$ be the section defined by $\lambda_f$ and let $\bar{\lambda}$ be its image in $(\mathcal{H}^{2k}/F^k\mathcal{H}^{2k})(U)$, where $ F^k\mathcal{H}^{2k} =\mathcal{H}^{2k,0}\oplus \mathcal{H}^{2k-1,1}\oplus \dots \oplus \mathcal{H}^{k,k}$.

\begin{dfn}[Local Noether-Lefschetz locus]\label{localNL} $NL^{k,\beta}_{\lambda,U}:=\{g\in U \mid \bar{\lambda}_{g}=0\}$. \end{dfn}
 
When the choice of $\beta$ and $k$ is clear from  the context we will denote the Noether-Lefschtez locus simply by $NL_{\lambda,U}$.

\begin{thm}$T_{[f]}\big(NL_{\lambda,U}\big)\cong E(f)^{\beta}$, where 
$$E(f)=\Bigl\{K\in S^{\bullet} \mid  \sum_{i=1}^b \lambda_i \int_{\operatorname{Tub}\gamma_i}  \frac{KR\Omega_0}{f^{k+1}}=0 \;for\;all\; R\in S^{N-\bullet}\Bigr\} ,$$
 $N=(k+1)\beta-\beta_0$ and $\operatorname{Tub}(-)$ is the adjoint of the residue map. Moreover, $E(f)$ contains $J(f)$, the Jacobian ideal of $f$.
\end{thm}

The  $\phi$-degree   for $\phi\in \operatorname{Hom}^+_\Z(\operatorname{Cl}(\Pj^{2k+1}_{\Sigma}),\Q)$ of a subvariety  in $\Pj^{2k+1}_{\Sigma}$ is defined in Subsection \ref{phidegree} via elimination of variables for $I_V$;   for $\Pj^r$ it coincides with the classical notion of degree.
We shall prove the following preliminary result.

\begin{prop}For every positive 
$\epsilon>0$ there is $ \delta>0$ such that for every $n\geq \frac{1}{\delta}$ and $d\in [1,n\delta]$,  if $\codim NL^{}_{\lambda,U}\leq d\frac{n^k}{k!}$ where $\phi(\beta)=n$,  then  for every element of $NL^{}_{\lambda,U}$  there is  a $k$-dimensional subvariety $V$ whose $\phi$-degree is less than or equal to $(1+\epsilon)d$.
\end{prop}

 \begin{dfn}Let $I'$ be an ideal in the Cox ring $S$ of $\Pj_{\Sigma}^{2k+1}$, 
and $\Lambda'\in (S^{N'})^\vee$   a linear functional   that  induces an isomorphism $(S/I')^{N'}\simeq \C$.
We say that $\alpha$ is regular for $I'$ if $\Lambda'$ defines an inclusion $(S/I')^{\alpha} \hookrightarrow \left((S/I')^{N'-\alpha}\right)^{\vee}$ and we say that $\alpha$ is biregular for $I'$ if $\Lambda'$ defines an isomorphism.
\end{dfn}
An  ideal for which  the  study of  regular classes is nontrivial  is the toric Jacobian ideal $J_0(f)$ (see equation \eqref{ToricJacobianEq}).

In the previous Proposition, let $V$ be given by the vanishing of $J=(F_0,\dots,F_k:F)\subset E(f)$,
let $Z$ be an irreducible $k$-dimensional subvariety of $V$ and let $Z'\subset V$ be a smallest subvariety such that $I_V=I_Z\cap I_{Z'}$. Set $\deg_{\phi}Z:=\phi(\kappa)$ and $\deg_{\phi}Z':=\phi(\kappa')$. 

The following is  our main result.

\begin{thm} For every positive 
$\epsilon$ there is  positive $ \delta$ such that for every $n\geq \frac{1}{\delta}$ and $d\in [1,n\delta]$, if:
\begin{itemize}
    \item $\sum_{i=0}^k \deg (F_i)$ is a regular class for $(F_0,\dots, F_{2k+1})$ where $F_{k+1},\dots F_{2k+1}$  are $k$-generators of $J_0(f)$;
    \item  $\kappa+2\kappa'$ is a biregular class for every pair $(Z,Z')$ where $Z$ is an irreducible subvariety of $V$ and $Z'$ is the smallest subvariety $Z'\subset V$ such that $I_V=I_Z\cap I_{Z'}$;
    \item $\codim NL^{}_{\lambda,U}\leq d\frac{n^k}{k!}$ where $\phi(\beta)=n$,
    
\end{itemize}
 then   every element of $g\in NL^{}_{\lambda,U}$ such that $E(g)^{\leq   \upsilon}\iso E(f)^{\leq \upsilon}$ {\it contains} a $k$-dimensional subvariety  whose $\phi$-degree is less than or equal to $(1+\epsilon)d$. 
\end{thm}

\begin{cor}  For every positive 
$\epsilon$ there is  positive $ \delta$ such that for every $n\geq \frac{1}{\delta}$ and $d\in [1,n\delta]$,  if $\Pj^{2k+1}_{\Sigma}$ has Picard rank 1 and $\codim NL^{}_{\lambda,U}\leq d\frac{n^k}{k!}$ where $\phi(\beta)=n$,  then  every element of $NL^{}_{\lambda,U}$ contains a $k$-dimensional subvariety whose $\phi$-degree is less than or equal to $(1+\epsilon)d$.
 \end{cor}

\smallskip
\textbf{Acknowledgements.}  This research was partly supported by the  PRIN project 2022BTA242  ``Geometry of algebraic structures: moduli, invariants, deformations,''
by GNSAGA-INdAM and by FAPESP through the grants  2017/22091-9, 2019/23499-7 and 2023/01360-2.

\bigskip
    \section{Toric Green's and Macaulay's theorems}

After fixing a positive integer $n$, any other   positive integer $c$ can be written in the form $$c=\binom{k_n}{n}+\dots +\binom{k_{1}}{1},$$ 
with  $k_n>k_{n-1}>\dots >  k_{1}\geq  0$,  see \cite{Green89}. This is called the $n$-th Macaulay decomposition of $c$.  
One adopts the convention that $\binom{k}{\ell} =0 $ if $k<\ell$ so that
one can also write
 $$c=\binom{k_n}{n}+\dots +\binom{k_{\delta}}{\delta},$$
where $\delta = \min \{m \,\vert\,k_m\ge m\}$.
 
Let $c$ be the codimension of a linear subsystem $W\subseteq H^0(\Pj^r,\cO_{\Pj^r}(n))$,  and let  $W_H\subseteq H^0(\cO_H(n))$ be the restriction of $W$  to a general hyperplane $H$ of codimension $c_H$. Then the classical restriction theorem says that  
 \begin{equation}\label{classicrestr}c_H\leq c_{<n>}, \end{equation}
where
$$ c_{<n>}:=\binom{k_n-1}{n}+\cdots +\binom{k_{1}-1}{1}.$$

We note two elementary properties of the function $\psi: c\mapsto c_{<n>}$:

\begin{enumerate}
\item[(A)]\label{A} If $c'\leq c$, then $c'_{<n>}\leq c_{<n>}$, i.e., the map $\psi$ is non-decreasing;
\item[(B)] if $k_{\delta}>\delta$ then $(c-1)_{<n>}<c_{<n>}$,  i.e., the map $\psi$ is increasing {when $k_{\delta}>\delta$.}
\end{enumerate}

{We generalize the inequality \eqref{classicrestr} to a simplicial  projective toric variety $\Pj_{\Sigma}$ using induction on  degree and dimension. 
Given a divisor $D$ in $\Pj_{\Sigma}$ and a linear system $W\subseteq H^0(\cO_{\Pj_{\Sigma}}(D))$,
we will denote by $c_D$ the codimension of $W_D$ in $H^0(\cO_D(D))$. The next two lemmas are a preparation
for the proof of the restriction theorem  (Theorem \ref{restriction thm}).}

Before we start let us introduce some definitions.
Let $\mathcal{L} = \langle x_1, . . . , x_m\rangle $  be the $\C$-linear subspace generated by the classes of the variables in $S$,
 the Cox ring of  $\Pj_{\Sigma}$.  Note that $\mathcal{L}$ has a natural  
grading $\mathcal L = \oplus_{i=1}^s \mathcal L_i$ induced by the  grading of  $S$.

\begin{dfn}\label{linear section}   We call 
any element $l\in \mathcal{L}_i$ for some $i\in \{1,\dots, s\}$ a linear section, and its corresponding  divisor $D$ a linear divisor. We say that a divisor is $\phi$-linear if $\phi([D])=1$ for a $\phi\in \operatorname{Hom}_\Z^+ (\operatorname{Cl}(\Pj^{2k+1}_{\Sigma}),\Q)$. 
    \end{dfn}

\begin{prop}\label{lineartoric} If  $l=\sum_{i=1}^k a_i x_i$ ($a_1\neq 0$) is linear section, its zero locus  is a simplicial toric variety. Moreover $\mathbb{V}(l)\simeq \mathbb{V}(x_1)$.
\end{prop}

\begin{proof} Let $l=\sum_{i=1}^k a_i x_i$ ($a_1\neq 0$) be a linear section and let $\psi$ be the graded automorphism of $S$ defined as $$(x_1,\dots x_k, x_{k+1},\dots x_m)\mapsto (\sum_{i=1}^k a_i x_i, x_2, \dots, x_m)$$ which   is clearly an isomorphims of graded $\C$-algebras. By Theorem 4.2 in \cite{Cox1993} there exists an automorphism $\tilde{\psi}$ inducing an isomorphism  $\mathbb{V}(l)\xrightarrow[\tilde{\psi}]{\sim}\mathbb{V}(x_1) $ and by Proposition 3.2.7 in \cite{CoxLittleSchenck}, $\mathbb{V}(x_1) $ is a simplicial toric variety.  
\end{proof}
\begin{rmk} The above result for $\Pj^r$ says that every  hyperplane  is isomorphic to $\Pj^{r-1}$.
     \end{rmk}

\begin{lma}  Let $D$ be  a general linear divisor in $\Pj_{\Sigma}$,
$W\subseteq H^0(\cO_{\Pj_{\Sigma}}(D))$   a linear subsystem,    and let $W_{D}\subseteq H^0(D,\cO_{\Pj_{\Sigma}}(D)) $ be the restriction of $W$ to $D$. Then 
$$c_D\leq c_{<1>}=\codim(W,H^0(\cO_{\Pj_{\Sigma}}(D)))-1.$$  
\label{L1}\end{lma}

\begin{proof} 
Taking cohomology  in the fundamental short exact sequence of the divisor $D$ and as every simplicial toric variety has zero irregularity,  $H^1(\cO_{\Pj_{\Sigma}})=0$,  we obtain
$$0\to H^0(\cO_{\Pj_{\Sigma}})\to H^0(\cO_{\Pj_{\Sigma}}(D))\to H^0(\cO_D(D))\to0  ,  $$
so that 
\begin{equation}\label{giovanni}
h^0(\cO_{\Pj_{\Sigma}}(D))=h^0(\cO_{\Pj_{\Sigma}})+h^0(\cO_D(D))=1+h^0(\cO_D(D)).
\end{equation}

Denoting by ${\rm pr}$ the projection $W\xrightarrow{{\rm pr}} W_D$ one has
\begin{equation}\label{andrea}
\dim W=\dim\ker {\rm pr}+\dim W_D.  
\end{equation}
so that subtracting \eqref{andrea} from \eqref{giovanni} we have  
$$\codim W=\codim W_D+1-\dim\ker {\rm pr}.$$
Now, if $s_D$ is a section in $H^0(\cO_{\Pj_{\Sigma}}(D))$ such that $D$ is the zero locus of $s_D$,  then 

$$\ker r=\{w\in W \mid {\rm div}_0(w)=\lambda D ,\ \  \lambda\in \C  \}, $$
and taking $s_D\notin W$, {$D$ is general,}  we conclude that $\ker {\rm pr}=\{0\}$.
\end{proof}

\begin{lma} Let  $D$ be  a linear divisor on $\Pj^1,$
let  $W\subset H^0(\cO_{\Pj^1}(\beta))$  be a linear subsystem with $\dim W\geq 1$, and let $W_D\subset H^0(\cO_{D}(\beta))$ be its restriction. Then $c_D=0$ and moreover
$$c_D\leq c_{<n>} .$$
\label{L2}\end{lma}

\begin{proof} Since $D$ is a linear divisor $\operatorname{dim}W_D= \operatorname{dim} H^0(\cO_D(\beta))\simeq \C $, i.e., $c_D=0$. 
\end{proof} 

The previous two Lemmas are  the base cases for the double induction in the proof of the following Theorem.

 \begin{thm}[Toric Green theorem]\label{restriction thm} Let ${\Pj_{\Sigma}}$ be  a projective simplicial toric variety and fix a element $\phi\in \operatorname{Hom}_{\Z}^+(\operatorname{Cl}(\Sigma), \Q)$. Let $W\subseteq H^0({\Pj_{\Sigma}},\cO_{\Pj_{\Sigma}}(\beta))$, be a linear subsystem with $\beta$ an ample   Cartier class such that $\phi(\beta)=n\in\Z$, and assume that $D$   is a  $\phi$-linear general divisor;  let $W_D\subseteq H^0(D,\cO_D(\beta))$ be its restriction to $D$.  Then 
$$c_D\leq c_{<n>}.$$   
\end{thm}

\begin{proof} Let $l_n,\dots ,l_{\delta}$ be the coefficients of the $n$-th Macaulay decomposition of $c_D$. The inequality in the statement  is equivalent to 
$$\binom{l_n+1}{n}+\binom{l_{n-1}+1}{n-1}+\dots +\binom{l_{\delta}+1}{\delta}<c. $$
By contradiction, and recalling that $\binom{l+1}{n}=\binom{l}{n}+\binom{l}{n-1}$, we have
$$c\leq \binom{l_n}{n}+\binom{l_n}{n-1}+\dots +\binom{l_{\delta}}{\delta}+\binom{l_{\delta}}{\delta-1}$$
or equivalently 
\begin{equation} \label{3}
c-c_D\leq \binom{l_n}{n-1}+\dots +\binom{l_{\delta}}{\delta-1}.
\end{equation}
From   the  exact sequence  
$$0\to W(-D)\to W\to W_D\to 0$$
one has
\begin{equation}\label{4}
    \dim W=\dim W_D+\dim W (-D).
\end{equation}
 One has $H^1(\cO_{\Pj_{\Sigma}}(\beta-\eta))=0$ by Mustat\v a vanishing theorem (Corollary 2.5 in \cite{mustata}), so that
$$0\to H^0(\cO_{\Pj_{\Sigma}}(\beta-\eta)\to H^0(\cO_{\Pj_{\Sigma}}(\beta))\to H^0(\cO_D(\beta))\to  0$$
and thus

\begin{equation}
h^0(\cO_{\Pj_{\Sigma}}(\beta))=h^0(\cO_{\Pj_{\Sigma}}(\beta-\eta)+h^0(\cO_D(\beta)).    \label{5}
\end{equation}
Then \eqref{4} minus \eqref{5} yields
$$c=c_D+\codim W(-D).$$

Taking $D'$ a general linear divisor of the toric variety $D$, we are within the same assumptions of the theorem on $D$ as the ambient.

Now, we have the short exact sequence
$$0\to W_{D}(-(D\cap D'))\to W_D\to {W_D}_{|D'}\to 0$$
which gives  
$$ c_D= \codim {W_D}_{|D'}+\codim  W_{D}(-(D\cap D')) $$
Note that $W(-D')_{D}\subset W_{D}(-(D\cap D'))$, hence
$$ c_D\leq \codim {W_D}_{|D'}+\codim W(-D')_{D} $$
By induction on the degre  and the dimension of the ambient variety, we may assume that the  theorem holds true for
$W_D$ and $W(-D)$;  note that Lemmas \ref{L1} and \ref{L2} provide the induction base.

Applying the theorem to $W_{D}$ and $W(-D)$ we get 
\begin{itemize}
\item $(c_{D})_{|D'}\leq (c_D)_{<n>}=\binom{l_n-1}{n}+\dots +\binom{l_{\delta}-1}{\delta}$ 
\item $(c-c_D)_{|D'}\leq (c-c_{D})_{<n-1>} $
\end{itemize}
Adding the two inequalities and keeping in mind that $D'\sim D$ we have
$$c_{D'}=c_{D}\leq (c_D)_{<n>}+ (c-c_{D})_{<n-1>},$$
and by \eqref{3} and property (A)
$$(c-c_D)_{<n-1>}< \binom{l_n-1}{n-1}+\dots +\binom{l_{\delta}-1}{\delta-1},$$
so that
$$c_{D}< \binom{l_n-1}{n}+\dots +\binom{l_{\delta}-1}{\delta}+\binom{l_n-1}{n-1}+\dots +\binom{l_{\delta}-1}{\delta-1}=c_{D}$$
which is a contradiction.
\end{proof}

Let $W\subset H^0(\cO_{\Pj_{\Sigma}}(\beta))$ be  a subsystem and let $k_n,k_{n-1},\dots k_{\delta}$ be the 
 Macaulay coefficients of its codimension  $c$; 
  let $W_1$ be the image of the multiplication map $W\otimes H^0(\cO_{\Pj_{\Sigma}}(\eta))\to  H^0(\cO_{\Pj_{\Sigma}}(\beta+\eta)))$, and $c_1$    the codimension of its image. Let us denote 
$$c^{<n>}:=\binom{k_n+1}{n+1}+\dots +\binom{k_\delta+1}{\delta+1}.$$
 This   has the following elementary properties.
\begin{itemize}
\item If $c'\leq c$ then $c'^{<n>}\leq c^{<n>}$, i.e, the map $c\mapsto c^{<n>}$ is non-decreasing;
\item $(c+1)^{<n>}=\left\lbrace\begin{array}{lc}
c^{<n>}+k_1+1& if\; \delta=1\\
c^{<n>}+1& if\; \delta>1 ;
\end{array}\right.$
\item $c^{<n>}+c'^{<n>}\leq (c+c')^{<n>}$;
\item $c^{<n+1>}\leq c^{<n>}$.
\end{itemize}

The first two properties are clear and for the last two see Lemmas 2 and 3 in \cite{Otwinowska2003}.

\begin{thm}[Toric Macaulay's theorem]\label{ToricGreen} $c_1\leq c^{<n>}$
\end{thm}

\begin{proof}
Let $l_{n+1},l_n,\dots l_{\delta}$ be the $(n+1)$-th  Macaulay coefficients of $c_1$;  then 
$$ (c_{1})_{D}\leq c_{<n>}=\binom{l_{n+1}-1}{n+1}+\dots +\binom{l_{\delta}-1}{\delta} $$
and by the sequence obtained by restriction it follows that 
$$c_1\leq c+(c_1)_{D} $$
so that
$$\binom{l_{n+1}-1}{n}+\dots +\binom{l_{\delta}-1}{\delta-1}\leq c$$
and then
$$\binom{l_{n+1}}{n+1}+\dots +\binom{l_{\delta}}{\delta}=c_1\leq c^{<n>}.$$

\end{proof}

\section{Hypersurfaces in toric varieties and Cox-Gorenstein ideals}
In this section we introduce the notion of Gorenstein ideal for Cox rings.

\subsection{Cox-Gorenstein ideals}
The Cox ring $S$ of a complete simplicial toric variety $\Pj_\Sigma$  is graded over the effective classes in the class group $\operatorname{Cl}(\Pj_{\Sigma})$ 
 $$S = \sum_{\alpha \in \operatorname{Cl}(\Pj^{2k+1}_{\Sigma})} S^\alpha, \qquad
S^\alpha =  H^0(\Pj_\Sigma,\cO_{\Pj_\Sigma}(\alpha))$$
 (see e.g.~\cite{Cox1993}).  After \cite{BruzzoGondimHolandaMontoya}, 
 we   give a definition of {\em Cox-Gorenstein ideal} which generalizes to toric varieties the definition
 as given by Otwinowska in \cite{Otwinowska2003} for projective spaces.

 \begin{dfn} \label{lambdator} 
 A graded ideal $I$ of $S$ 
 is said to be
 a Cox-Gorenstein ideal of socle degree  $N\in\operatorname{Cl}(\Pj^{2k+1}_{\Sigma})$  if
 \begin{enumerate} 
 \item the quotient $R=S/I$ is Artinian; 
\item $\dim_\C R^N = 1$; 
\item for every homogeneous  class $\alpha\in\operatorname{Cl}(\Pj_{\Sigma})$, either 
the natural bilinear morphism (called ``Poincar\'e duality'') \begin{equation}\label{pair}R^\alpha \times R^{N-\alpha } \to R^N\simeq \C \end{equation} is nondegenerate,
or $R^\alpha=R^{N-\alpha}=0$.  \end{enumerate} 
\end{dfn}

\begin{prop} 
 An artinian graded ideal $I$ of $S$ 
 is  
  Cox-Gorenstein   of socle degree  $N $  
  if and only there exists a linear functional $\Lambda \in (S^N)^\vee$ such that for every $\alpha \in \operatorname{Cl}(\Pj_{\Sigma})$
\begin{equation}\label{Lambda} I^\alpha = \{ P\in S^\alpha \,\vert\, \Lambda(PQ) =0 \ \text{for all} \ Q\in S^{N-\alpha}\}.\end{equation}
\end{prop}

\begin{proof} If $I$ is Cox-Gorenstein, the functional $\Lambda\in (S^N)^\vee$ defined by the composition $S^N \to R^N \to \C$ satisfies condition \eqref{Lambda}. Indeed, \eqref{Lambda} certainly holds if $P\in I^\alpha$. On other hand, if \eqref{Lambda} holds for some $P\in S^\alpha$ then it implies $u\cdot v =0$
for all $v\in R^{N-\alpha}$ and $u=\bar P$, so that $u=0$ and $P\in I^\alpha$.

Vice versa, such a $\Lambda$
defines a pairing $\tilde\Lambda\colon R^\alpha\times R^{N-\alpha}\to \C$. If $u\in R^\alpha$ is such that $\tilde\Lambda(u,v)=0$ for all $v\in R^{N-\alpha}$, then $\Lambda(P,Q)=0$ for all $Q\in S^{N-\alpha}$
if $\bar P =u$, so that $P\in I^\alpha$, i.e., $u=0$, and $\tilde\Lambda$ is nondegenerate. \end{proof}

\begin{prop}\label{lambda}Let $I'\subset I$ be  Cox-Gorenstein ideals in $S$  with socle degree $N'$ and $N$ respectively. Let $\Lambda'$, $\Lambda$ be the functionals associated with $I'$ and $I$. Then there exists $F\in S^{N'-N}$ such that $\Lambda(Q)=\Lambda'(QF)$ for all $Q\in S^N$.
\label{propACG}
\end{prop}

\begin{proof} Dualizing the exact sequence 
$$ 0 \to I'^{N} \to S^{N} \to S^{N}/I'^{N} \to 0$$
we obtain
$$ 0 \to (S^{N}/I'^{N})^\vee \to  (S^{N})^\vee \to  (I'^{N})^{\vee} \to 0.$$
 Since $\Lambda$ is zero on $ I'^{N}$ so it lies in $(S^{N}/I'^{N})^\vee$. The latter is isomorphic to $S^{N'-N}/I^{N'-N}$ via $\Lambda'$, so that there is a polynomial $F\in S^{N'-N}$ satisfying the required condition.
\end{proof}

\begin{prop}\label{F}  Let $I'$ be a Cox-Gorenstein ideal  of $S$ with socle degree $N'$;   for any  $ F\notin I'$ with $\deg(F)=N'-N \ge 0 $ the ideal  $(I':F)$ is a Cox-Gorenstein ideal of socle degree $N$.
    \end{prop}
    \begin{proof} If $\Lambda'\in (S^{N'})^\vee$ is the functional associated to $I'$, define $\Lambda\in (S^N)^\vee$ as $\Lambda(Q) = \Lambda'(QF)$. It is clear that $\Lambda(PQ)=0$ if
    $P\in I^\alpha$ and $Q\in S^{N-\alpha}$. On the other hand, if
    $P\in S^\alpha$ and $\Lambda(PQ)=\Lambda'(PFQ)=0$ for all $Q\in S^{N-\alpha}$ then $PF \in (I')^{N'-N+\alpha}$ so that $P\in I^\alpha$.
    \end{proof}

\subsection{Regular and biregular classes} It is quite  common to  meet pairs of ideals $I' \subset I$ of  $S$, where    $I$ is Cox-Gorenstein but $I'$ is not,  although $I'$ is associated with  a linear functional $\Lambda':S^{N'}\to \C$. Thus it is natural to ask for which degrees $I'$ behaves as a Cox-Gorenstein ideal, i.e., for which degrees $\alpha\in \operatorname{Cl}(\Pj^{2k+1}_{\Sigma})$,  $\Lambda'$ defines an inclusion, surjection or bijection of $(S/I')^{\alpha }$ with $\left((S/I')^{N'-\alpha}\right)^{\vee}$. 

 \begin{dfn}
Let $I'$ be an ideal in the Cox ring $S$ of $\Pj_{\Sigma}^{2k+1}$, 
and $\Lambda'\in (S^{N'})^\vee$   a linear functional   that  induces an isomorphism $(S/I')^{N'}\simeq \C$.
We say that $\alpha\in \operatorname{Cl}(\Pj^{2k+1}_\Sigma)$ is regular for $I'$ if $\Lambda'$ defines an inclusion $(S/I')^{\alpha} \hookrightarrow \left((S/I')^{N'-\alpha}\right)^{\vee}$ and we say that $\alpha$ is biregular for $I'$ if $\Lambda'$ defines an isomorphism or equivalently, 
$$I'^{\alpha}=\{ g\in S^{\alpha}\mid \Lambda'(gh)=0\, \text{for all}\ h\in S^{N'-\alpha} \}.$$ 
 \end{dfn}
 So, the condition that  every class is biregular for $I$  is equivalent to $I$ being  Cox-Gorenstein. 

\begin{rmk} 

If $I'=(f_0,\dots,f_r)$ with $\deg(f_i)$ big and nef,  and $r$ the dimension of the ambient toric variety,  then there exists $\Lambda'\in (S^{N'})^{\vee}$ where $N'=\sum_{i=0}^r \deg(f_i)-\beta_0 $, see Corollary 2.5 and Remark 2.6 in \cite{CoxDickenstein2005}.
    \end{rmk}

\begin{cor}\label{biregular} Let $I$ be a Cox-Gorenstein ideal with socle degree $N$  and let  $I'\subset I$ be an ideal  with  a linear functional  $\Lambda'\in (S^{N'})^\vee$    that  induces an isomorphism $(S/I')^{N'}\simeq \C$. If $N'-N\in \operatorname{Cl}(\Pj^{2k+1}_{\Sigma})$ is biregular for $I'$, then for any other $\alpha\in \operatorname{Cl}(\Pj^{2k+1}_{\Sigma})$  biregular for $I'$, one  has $I^{\alpha}=(I':F)^{\alpha}$.
\end{cor}

The following theorems provide a way to check, at least in some cases, if a class $\alpha\in \operatorname{Cl}(\Pj^{2k+1}_{\Sigma})$ is regular.

 \begin{thm}\label{regularthm}
Let $f_0,\dots, f_r$ be homogeneous polynomials, with $\deg(f_i)=\alpha_i$ ample classes,  that  do not vanish simultaneously. Then there exists $\Lambda'\in (S^{N'})^{\vee}$, with $N'=\sum_{i=0}^r \alpha_i-\beta_0$,  such that  every class $\alpha\in \operatorname{Cl}(\Pj^{2k+1}_{\Sigma})$  with 
$$H^q\left(\cO_{\Pj_{\Sigma}}(\alpha-\sum_{j=1}^{q+1}\alpha_{i_j} )\right) =0$$
for all $q=1,\dots r-1$  and $0\leq i_1<\cdots <i_{q+1}\leq r$,  is regular.
    \end{thm}

\begin{proof} It follows from Theorem 1.3 in \cite{Villaflor_Loyola_2024}.
    \end{proof}

\begin{thm} \cite[Thm.~6.3]{BruzzoGondimHolandaMontoya} 
Let $f_0,\dots, f_r$ be homogeneous polynomials, with $\deg(f_i)=\alpha_i$ nef and big classes, that  do not vanishing simultaneously.  If the Picard number  of $\Pj_{\Sigma}$ is 1 then each class $\alpha\in \operatorname{Cl}(\Pj^{2k+1}_{\Sigma})$ is biregular.
    \end{thm}

Examples of  ideals  that  are not Cox-Gorenstein     but   have some biregular classes  may be given in terms of {\em toric Jacobian ideals.}
For every ray $\rho\in\Sigma(1)$  denote by $v_\rho$ its minimal rational generator,
and by $x_\rho$ the corresponding variable in the Cox ring.  Let  $r =\dim\Pj^r_\Sigma$, and   denote
by $s=\#\Sigma(1)$ the number of rays. 
Given $f\in S^\beta$ one defines its {\em toric Jacobian ideal} as 

\begin{equation}\label{ToricJacobianEq}
J_0(f) = \left( x_{\rho_1} \frac{\partial f }{\partial  x_{\rho_1}}, \dots,  x_{\rho_s} \frac{\partial f }{\partial  x_{\rho_s}} \right).    
\end{equation}

We recall from \cite{BatyrevCox} the   definition of nondegenerate hypersurface and some properties (Def.~4.13 and Prop.~4.15).

\begin{dfn} Let $f\in S^\beta$, with $\beta$ an ample Cartier class. The associated hypersurface $X_f \subset \Pj^r_\Sigma$ is nondegenerate if for all $\sigma\in\Sigma$ the affine hypersurface $X_f\cap O(\sigma)$ is a smooth codimension one subvariety of the orbit $O(\sigma)$ of the action of the torus
$\mathbb T^r$. 
\end{dfn}

\begin{prop} \begin{enumerate}  \item Every nondegenerate hypersurface is quasi-smooth.
\item If $f$ is generic then $X_f$ is nondegenerate.
\end{enumerate} 
\end{prop}

We collect here, with some changes in the terminology, some results that are already 
 contained in Prop.~5.3 of \cite{CATTANI_COX_DICKENSTEIN_1997}.
 
\begin{prop}\label{ToricJacobian} Let $f\in S^\beta$, and let $\{\rho_1,\dots,\rho_r\}\subset\Sigma(1)$ be such that
$v_{\rho_1},\dots,v_{\rho_r}$ are linearly independent.  \begin{enumerate}  \item The toric Jacobian ideal of $f$ coincides with the ideal 
$$ \left( f, x_{\rho_1} \frac{\partial f }{\partial  x_{\rho_1}}, \dots,  x_{\rho_r} \frac{\partial f }{\partial  x_{\rho_r}} \right).$$
\item The following conditions are equivalent:
\begin{enumerate} \item $f$ is nondegenerate; \item the polynomials $x_{\rho_i} \frac{\partial f }{\partial  x_{\rho_i}}$, $i=1,\dots,s$, do not vanish simultaneously on $X_f$; \item
the polynomials  $f$ and $x_{\rho_i} \frac{\partial f }{\partial  x_{\rho_i}}$, $i=1,\dots,r$, do not vanish simultaneously on $X_f$.
\end{enumerate} \end{enumerate} 

\end{prop}

\begin{prop} Each class of the form ${s\beta+t\beta_0}$ is biregular for $J_0(f)$  for any $s\in \Z$ and  $t\in \{-1,0,1,2\}$.
\end{prop}  

\begin{proof}
It is a direct consequence of  Proposition 3.3 in \cite{Villaflor_Loyola_2024}.
\end{proof}

\bigskip
\section{Some  applications of toric Macaulay's theorem}

In this section we prove some applications of  Green's theorem to Cox-Gorenstein ideals.
This generalizes some of the results by Otwinowska in \cite{Otwinowka2002,Otwinowska2003}  to the more general setting
of odd-dimensional toric varieties, as opposed to odd-dimensional projective spaces, which is the case considered by her. Firstly we introduce some preliminaries related with the $\phi$-degree of a subvariety $V\subset \Pj_{\Sigma}$ in a $r$-dimensional toric variety.

\subsection{$\phi$-degree of a subvariety for $\mathbf{\phi\in \operatorname{Hom}_{\Z}^+(\Cl(\Sigma),\Q)}$}\label{phidegree}

Let $\Pj^{k+1}_{\Sigma'}\subset \Pj_{\Sigma}
$ be a toric subvariety of $\Pj_{\Sigma}$,
obtained by setting to zero a number of variables in the Cox ring of $\Pj_{\Sigma}$;
this corresponds to picking up a number of rays of
$\Sigma$, and projecting the fan $\Sigma$ to the quotient by the subspace generated by those rays.  Let $S_P=\C[x_1,\dots, x_t]\subset S$ be the Cox ring of $\Pj^{k+1}_{\Sigma'}$;
if  $Z\subset \Pj_{\Sigma}$ is an irreducible $k$-dimensional subvariety, one can arrange that $P=\{x_1=\dots=x_{t}=0\}$ does not intersect $Z$. Then we have a well defined projection map 
$$\pi:Z\to \Pj^{k+1}_{\Sigma'},$$
and by the  fiber dimension theorem and since $Z$ is irreducible, $\dim Z=\dim \pi(Z)$. The following Lemma gives the ideal associated to $\pi(Z)$.

\begin{lma} The variety $\pi(Z)$ is defined by the elimination ideal $I_Z\cap S_P$, i.e., 
$\pi(Z)=\mathbb{V}(I_Z \cap S_P). $
\end{lma}
\begin{proof} By $\S$ 2.2, Theorem 3  in \cite{CoxLittleOshea} we have the result for the affine case and by the toric ideal correspondence, Proposition 5.2.4 in \cite{CoxLittleSchenck}, we get the result (about the elimination ideal, see also \cite{Sturm-elim}).
    \end{proof}

 By the previous Lemma $I_Z\cap S_P\subset S_P$ is a principal ideal $(f_{Z,P})$;
the corresponding subscheme $D_{Z,P}$  in $\Pj_{\Sigma}$ is 
called a \emph{cone}  with base $Z$ and vertex $P$.
 
 In Definition \ref{linear section} we introduced the notion of a linear section.  
We call a subvariety $P \subset \Pj_{\Sigma}$ {\it linear} if is it 
defined by the intersection of linear sections.
By Proposition \ref{lineartoric} applied inductively, there are only finitely many linear subvarieties (up to isomorphism).

\begin{dfn}Let $Z\subset \Pj_{\Sigma}$ be an irreducible $k$-dimensional subvariety and let $\phi\in \operatorname{Hom}_{\Z}^+(\operatorname{Cl}(\Sigma),\Q)$. We define the $\phi$-degree of $Z$ as
$$\deg_{\phi}(Z):=\max\{\phi([D_{Z,P}])\mid P\cap Z=\emptyset, \, P\, \text{linear}\},$$
and we denote by $\delta(Z):= \{[D_{Z,P}]
\mid P\cap Z=\emptyset, \, P\, \text{linear}\}$.

\end{dfn}
\begin{rmk}
When $\Pj_{\Sigma}=\Pj^r$, we recover the classical definition of degree  by taking $\phi$ the identity map and $P$ general. If $Z$ is reducible we define $deg_{\phi}(Z)$ as the maximum of the $\phi$-degrees of its irreducible components.
    
\end{rmk}

\begin{ex} Let $\Pj_{\Sigma}$ be the toric variety $\Pj^2\times \Pj^1$ with Cox ring $\C[x_1,x_2,x_3,x_4,x_5]$ and let $Z$ be an irreducible curve not containing $P=\{x_1=x_2=x_4=0\}$. Taking $\phi:\Z^2\to \Z$, $(a,b)\mapsto a+b$, we have 
$\deg_{\phi}Z=Z \cdot (\pi_1^*(H_1)+\pi_2^*(H_2)).  $
where $\pi_1, \pi_2$ are the natural projections and $H_1$ and $H_2$ the hyperplane sections of $\Pj^2$ and $\Pj^1$ respectively.

\end{ex}

\subsection{Applications}
 Let $\Pj_{\Sigma}$ be  an $r$-dimensional toric variety and  $\beta$   an ample class in it. Let  $W\subsetneq H^0(\cO_{\Pj_{\Sigma}}(\beta))$ be a linear subsystem, with base locus $Z$, and set $k = \dim Z$ and $d = \deg_{\phi}(Z)$.

\begin{lma} Assume that $H^1(\Pj_{\Sigma},I_Z(\beta +m\eta))=0$ for $m\gg 0$, where $\mathcal{I}_Z$ is the  ideal sheaf of $Z$. Then
$$\codim(W):=\dim S^{\beta}/W\geq \binom{n+k+1}{k+1}-\binom{n-d+k+1}{k+1},$$ where $n = \phi(\beta)$.
\end{lma}

\begin{proof}  Since $I_W\subset I_Z$, one has 
$\codim W\geq \codim I_Z^\beta$,  so it is enough to prove that the result holds for $\codim I^\beta_Z$. We will prove that by induction over $n$ and $k$.  {For $n=0$ the result is trivial.}
For $k=0$ and $n>0$ we need to show that $\codim I^\beta_Z\geq d$ if $n\geq d$ and $\codim I^\beta_Z\geq n+1$ if $n<d$. Taking cohomology in the exact sequence $$0\rightarrow \mathcal{I}_Z(\beta +m\eta))\rightarrow \cO_{\Pj_{\Sigma}}(\beta +m\eta))\rightarrow \cO_Z(\beta +m\eta))\rightarrow 0$$ 
we have
$$0\rightarrow H^0(\mathcal{I}_Z(\beta +m\eta)))\rightarrow H^0(\cO_{\Pj_{\Sigma}}(\beta +m\eta)))\rightarrow H^0(\cO_Z(\beta +m\eta)))\rightarrow H^1(\mathcal{I}_Z(\beta +m\eta)))\rightarrow \cdots $$
where by  hypothesis $H^1(\mathcal{I}(\beta +m\eta)))=0$ for $m\gg 0$. Thus 
$$\codim  I_Z^{\beta +m\eta}= 
h^0(\cO_{\Pj_{\Sigma}}(\beta +m\eta)))-h^0(\mathcal{I}_Z(\beta +m\eta)))=h^0(\cO_Z(\beta +m\eta)))=d.$$ 
If $n\geq d$ and assuming $\codim I^\beta_Z < d$ we have  
$$d=\binom{n}{n}+\dots+\binom{n-{(d-1)}}{n-(d-1)}=\underbrace{1+\dots+1}_{d-times},$$
applying the Green's theorem and using the fact  that the map $c\mapsto c^{<n>}$ is increasing, 
$$c_1\leq {c^{<n>}}<d \ \mbox{where} \ c_1=\codim I^{\beta+\eta}_Z;$$
repeating the same argument replacing $c$ with $c_1$ one has 
$$c_2\leq c_1^{<n+1>}\leq (c^{<n>})^{<n+1>}<d\ \mbox{where} \ c_2=\codim I^{\beta+2\eta}_Z, $$  
so that
$$c_m\leq (c^{<n>})^{<n+1>\dots <n+m-1>}<d$$
which implies $c_m\leq d-1$. This is a contradiction as $c_m=d$. Hence $\codim I^{\beta}_Z\geq d$ as we wanted.

Now, if $n<d$ and by contradiction if $\codim I^{\beta}_Z\leq n$ then $c<d$ and applying toric Green's theorem as before one gets $d\leq n$ which is a contradiction.

Now let us assume that the result is true for $n-1$ and $k$,  and $n$ and $k-1$. 

\textbf{Claim:} If  $D$ is general effective divisor of class $\eta$, the multiplication for a nonzero global section $s_D$ of  {$\cO_{\Pj_{\Sigma}}
(D)$}
$$\mu_D:S^{\beta-\eta}/I^{\beta-\eta}_Z\rightarrow S^{\beta}/I^{\beta}_Z$$
is injective. 

Since $D$ is general we may assume that $Z\cap D\neq Z$, i.e., $\mu_D\neq 0$. Now, if $\mu_D(f)=0$ then $f\cdot s_D=0$ and since $s_D\neq 0$ then $f=0$.
We have a well defined surjective restriction map $ S^{\beta}/I^{\beta}_Z\xrightarrow{r} S^{\beta}/I^{\beta}_{Z\cap D}$.  There is a short exact sequence 
$$0\rightarrow \ker r \xrightarrow{\mu_D} S^\beta/I^\beta_Z\xrightarrow{r} S^\beta/I^\beta_{Z\cap D}\rightarrow 0  .$$

It is clear that $\ker r$ contains $S^{\beta-\eta}/I^{\beta-\eta}_Z$. By the induction hypothesis, we have  

\begin{equation}\label{8}
\codim I^{\beta-\eta}_Z\geq \binom{n+k}{k+1}-\binom{n-d+k}{k+1}
\end{equation}
and
\begin{equation}\label{9}
\codim I^{\beta}_{Z\cap D}\geq \binom{n+k}{k}-\binom{n-d+k}{k};
\end{equation}
thus adding \eqref{8} and \eqref{9}
we get the result.
\end{proof}

\begin{cor}Let $W\subset H^0(\cO_{\Pj_{\Sigma}}(\beta))$ be a  linear subsystem whose base locus  $Z$ has dimension $k$ and is such that $\deg_{\phi}(Z) \ge  d$. Then for every $x\leq \min\{d,n\}$ one has
$$x\frac{(n-x)^k}{k!}\leq \codim W. $$
\end{cor}

\begin{proof}
Since $\binom{n+k+1}{k+1}-\binom{n-d+k+1}{k+1}= \sum_{j=1}^d\binom{k+1+n-j}{n-j+1}$ by applying the above lemma we get
\begin{multline} \sum_{j=1}^d\binom{k+1+n-j}{n-j+1}\geq \sum_{j=1}^d \frac{(k+1+n-j)\dots (k-(k-1)+1+n-j)}{k!} \\ \geq \sum_{j=1}^d\frac{(n-j)^k}{k!} 
\geq x\frac{(n-x)^k}{k!} \end{multline}
\end{proof}
We establish a pre-order in $N^1(\Pj_{\Sigma}) = \Pic(\Pj_{\Sigma})\otimes \mathbb Q$ by letting
$M \le N$ when $N-M$ is  effective.

\begin{prop}\label{majorante} For every $\epsilon_1>0$ there exists $\delta_1 >0$ such that for every $n\geq \frac{1}{\delta_1}$ and every real number $d\in[1,\delta_1n]$, if a Cox-Gorenstein ideal $I$ of socle degree $N$ satisfies 

\begin{itemize}\itemsep-2pt
\item $ 2\beta-\beta_0 \leq N$,  where  $\phi(\beta)=n$, 
\item $\codim I^{\beta}\leq d\frac{n^k}{k!}$,
\end{itemize}

then
\begin{enumerate}\itemsep-2pt
\item For every integer $i\in \{0,\dots , \lfloor \delta_1 n\rfloor\}$, one has
$$\codim I^{\beta-i\eta}\leq (1+\epsilon_1)d\frac{n^k}{k!} .$$
\item For every $i\in\{0,\dots n\}$, one has
$$\codim I^{\beta-i\eta}\leq 4^k d \frac{n^k}{k!}. $$
\end{enumerate}
\end{prop}

\begin{proof}
By   toric Green's theorem applied $i$ times one has

$$\codim I^{\beta+i\eta}\leq (\codim I^{\beta})^{<n> \cdots <n+i-1>} $$
Since $I$ is Cox-Gorenstein then  $\codim I^{\beta}= \codim I^{N-\beta}$ and $\codim I^{\beta-i\eta}= \codim I^{N-\beta+i\eta}$, and  thus
$$\codim I^{\beta-i\eta}\leq (\codim I^{\beta})^{<\phi(N)-n> \cdots <\phi(N)-n+i-1>},$$
and moreover by the properties of the functions $c\mapsto c^{<n>}$ and $n\mapsto c^{<n>}$ one has that for every $x\leq \phi(N)-n $
$$\codim I^{\beta-i\eta}\leq (\codim I^{\beta})^{<x> \cdots <x+i-1>}. $$

 The rest of the proof follows verbatim  the  proof of ``Proposition majorante''  in \cite{Otwinowska2003}.  
\end{proof}
For $I$ satisfying the assumptions of the previous Proposition we have:

\begin{cor} Let $\eta_1,\dots \eta_{\rho}$ be $\phi$-linear class divisors with $\rho={\rk}\operatorname{Cl}(\Pj^{2k+1}_{\Sigma})$. Then for every $i_1\dots i_{\rho}\in \{0,\dots,n\}$ and $\alpha$ of the form  $\beta-\sum_{j=1}^{\rho} i_j\eta_j$,
$$\codim I^{\alpha }\leq 4^{\rho k} d \frac{n^k}{k!}. $$
\end{cor}
 
\begin{dfn} Let $I\subset S$ be an ideal. For $i\in \{0,\dots,r\}$ we define 
$$\ell_i(I):=\{\gamma\in \operatorname{Cl}(\Pj^{2k+1}_{\Sigma})\mid \dim \mathbb{V}(I^{\gamma})\leq r-i-1\};$$
and for $\phi\in {\operatorname{Hom}_\Z^+(\operatorname{Cl}(\Pj^{2k+1}_{\Sigma}),\Q)}$
$$l_i^{\phi}(I):=\min \{\phi (\gamma)\in \N\cup \infty \mid \gamma\in \ell_i\} .$$
\end{dfn}
We let $\dim \emptyset = -1$, and $l_i=\infty$ when this number does not exist.
\begin{rmk}

\begin{itemize} 
\item We shall write $l_i(I)$ instead of $l_i^{\phi}(I)$.
   \item Note that $l_0(I)\leq \dots \leq l_{2k}(I)$.
   \item If $I$ is base point free in $\Pj_{\Sigma}$, then $l_{r}(I)\in \N$.
\end{itemize}
\end{rmk}

\begin{lma}\label{FundamentalLemma} For every $\epsilon_2>0$ there exists $\delta_2 >0$ such that for every $n\geq \frac{1}{\delta_2}$ and  $d\in [1,\delta_2 n]$, if a Cox-Gorenstein ideal $I\subset S$ with socle degree $N$ satisfies :
\begin{itemize}\itemsep=-2pt
\item $ 2\beta-\beta_0 \leq N$,  where $n=\phi(\beta)$,
\item $\codim I^{\beta}\leq d\frac{n^k}{k!}$, 
\item there exists $\gamma\in\ell_i(I)$ of the form $\beta-\sum_{j=1}^{\rho} i_j\eta_j$ such that $\phi(\gamma)=l_i(I)$
\end{itemize}
then 
$$\gamma\leq\epsilon_2 \beta \;\;\forall i\in \{0,\dots , r-k-1\},$$  

\end{lma}

\begin{proof} {Note that  it is enough to prove the Lemma for $i=r-k-1$,} so we apply the previous Proposition for $\epsilon_1=1$, and the Corollary for $x=1$.  Setting  $l=l_i(I)$ then we have 

$$\frac{(l-1)^{k+1}}{(k+1)!}\leq \codim I^{\gamma}\leq 4^{\rho k}d\frac{n^k}{k!} $$
so that
$$l\leq 1+\big(4^{\rho k}dn^k(k+1)\big)^{\frac{1}{k+1}}\leq \big(\frac{1}{n}+(4^{\rho k}(k+1)\frac{d}{n})^{\frac{1}{k+1}}\big)n\leq (\delta_2+(4^{\rho k}(k+1)\delta_2)^{\frac{1}{k+1}})n$$
and since $2\leq 2n\delta_2$ by assumption, adding and subtracting $2\delta_2$  in the right hand of the previous inequality one has
$$l\leq (3\delta_2+(4^{\rho k}(k+1)\delta_2)^{\frac{1}{k+1}})n-2.$$
So, given $\epsilon_2>0$, we take $\delta_2$ small enough to have $3\delta_2+(4^{\rho k}(k+1)\delta_2)^{\frac{1}{k+1}}<\min\{1,\epsilon_2\}$; then $l_i(I)\leq \epsilon_2n$ and  by the form of $\gamma$ and that for each $\pi_i:\operatorname{Cl}(\Pj^{2k+1}_{\Sigma})\to \Q $ the natural projection $\pi_i\leq \phi$, we get the result.
\end{proof}

\begin{rmk} The third condition in the previous Lemma is redundant when the toric variety has Picard rank 1.
    
\end{rmk}

The following Proposition will be the technical core of what follows.  

\begin{prop}\label{MajorProp}
 Let $\Pj_{\Sigma}^{2k+1}$ be an odd dimensional toric variety and let $\eta_1,\dots n_{\rho}$ be $\phi$-linear divisor classes, where $\rho=\rk \operatorname{Cl}(\Pj_{\Sigma}^{2k+1})$. For every $\epsilon>0$ there exists  $\delta>0$ such that for every integer $n>\frac{1}{\delta}$  and for every $d\in[1,\delta n]$, if a Cox-Gorenstein ideal $I\subset  S$  with socle degree $N$ satisfies:
\begin{enumerate}[label=\roman*)]\itemsep=02pt
\item  $N=(k+1)\beta-\beta_0 $ ;
\item  $I$ contains $(2k+2)$-polynomials with empty vanishing  and with the same ample degree $\beta$;
\item there exists $\gamma\in\ell_i(I)$ of the form $\beta-\sum_{j=1}^{\rho} i_j\eta_j$ such that $\phi(\gamma)=l_i(I)$,

\end{enumerate}
then there exists a  closed  subvariety $V\subset \Pj_{\Sigma}^{2k+1}$ of pure dimension $k$ and  $\phi$-degree less than or equal to $(1+\epsilon)d$.  
\end{prop}

\begin{proof}
 By the second and third assumption and the previous Lemma there exist $F_0,F_1,\dots F_{k}\in I^{\leq \gamma}$ with $\gamma\leq \epsilon_2\beta$, and  it is possible to find $(k+1)$-polynomials $F_{k+1},\dots , F_{2k+1}$, where $\deg(F_i)=\beta$ ( $i>j$), so that  the ideal  $(F_0,\dots F_r)\subset I$ is base point free in $\Pj^{2k+1}_{\Sigma}$ and there exists $\Lambda'\in S^{N'}$ with $N'= \sum_{i=0}^{r}\deg(F_i)-\beta_0$, and thus
 $$
N'=\sum_{i=0}^{r}\deg(F_i)-\beta_0\leq (k+1)\gamma+(k+1)\beta-\beta_0. 
$$
Then for any $F\notin(F_0,\dots, F_{2k+1})$ such that $\deg(F)\leq (k+1)\gamma$, we consider $J=\big((F_0,\dots F_{k}):F \big)$  which coincides with $(F_0,\dots, F_{2k+1}:F)$ in degree less than 
$$\beta-(k+1)\gamma\geq (1-(k+1)\epsilon_2)\beta.$$ 

 Now, since $(F_0,\dots, F_{k})$ are in complete intersection, their vanishing defines a subvariety $Y$   of dimension $k$; 
then  $V:=\mathbb{V}(J)$   also has dimension $k$. 

Let $\epsilon\in (0,1)$; we choose $\epsilon_1$ and $\epsilon_2$ such that 
\begin{equation}\label{EpsilonInequality}
    \frac{1+\epsilon_1}{(1-2\epsilon_2(k+1))^k} \leq 1+\epsilon
\end{equation}
Let $\delta_1$ from Proposition \ref{majorante} applied to $\epsilon_2$ and assume $\epsilon_2$ small enough such that for every $n\geq 4/\delta_1$,

$$\epsilon_2(k+1) \beta\leq \delta_1 \beta.$$

Now, let $\delta\leq \{\frac{\delta_1}{4}, \delta_2, \frac{\epsilon_2}{7}\}$ where $\delta_2$ is from Lemma \ref{FundamentalLemma} applied to $\epsilon_2$. 
  Set $\sum_{m=1}^{\rho}{j_m}\eta_m$ ($j_m\in \N^{+}$) the minimum of elements  such that $\sum_{m=1}^{\rho}{j_m}\eta_m\geq (k+1)\epsilon_2\beta$;  setting $l=\phi(\beta-\sum_m{j_m}\eta_m)$, we have for every $x\leq \min\{n,{\rm deg}_{\phi}V\}$ that 
$$x\frac{(l-x)^k}{k!}\leq \codim I^{\beta-\sum_m{j_m}\eta_m}\leq (1+\epsilon_1)d\frac{n^k}{k!}, $$
and thus
\begin{equation}\label{xinequality}
   -x\left(\frac{\lceil\epsilon_2(k+1)n \rceil+x}{n}-1\right)^k\leq (1+\epsilon_1)d, 
\end{equation}
so taking $x=\lceil\epsilon_2(k+1)n \rceil $ and applying inequality \ref{EpsilonInequality} we get $\lceil\epsilon_2(k+1)n \rceil\leq (1+\epsilon)d\leq 2d$ which contradicts the fact that $2d\leq 2 \delta n< \lceil \epsilon_2 n \rceil \leq \lceil \epsilon_2(k+1) n \rceil$. Thus $\deg_{\phi}V \leq \lceil \epsilon_2(k+1) n \rceil $ so we can  apply the inequality $\ref{xinequality}$ for $x=\deg_{\phi} V $  and use the inequality $\ref{EpsilonInequality}$ to get the result. 
\end{proof}

The  final part  of the previous proof implies that if $D_{V,P}$ is a divisor associated with a cone with base $V$, then $\phi([D_{V,p}])\leq \phi(\beta-\sum_m j_m \eta_m)$. Hence without loss of generality we can choose  $\epsilon_2$ small enough so that  
$$[D_{V,P}]\leq \beta-(k+1)\epsilon_2\beta\leq \beta-\sum_{m=1}^{\rho} j_m \eta_m.$$

\begin{cor}\label{FundamentalCor} There are cones  $F_0,\dots F_k$  with base $V$ such that 
whenever  $\alpha\in \operatorname{Cl}(\Pj^{2k+1}_{\Sigma})$ is a class strictly less than  
$\beta-(k+1)([D_{V,P}])$,
then
$(I_V)^{\alpha}$ coincides with $(F_0,\dots, F_{2k+1}:F)^{\alpha}$.

Moreover, if $\sum_i \deg F_i$ is regular for $(F_0,\dots, F_{2k+1})$, $F$ is the polynomial associated to $\Lambda$ and if $\alpha$ is a biregular class  for $(F_0,\dots, F_{2k+1})$, then  $(I_V)^{\alpha}=I^{\alpha}$.     
\end{cor}

\begin{proof}    $V$ is of pure dimension $k$,  and by its construction one can find $(k+1)$-polynomials   ${f}_{V,P_i}$ in complete intersection which define    cones with base $V$ and vertex $P_i$  (see Subsection \ref{phidegree})
  such that  $\deg({f}_{V,P_k})\leq [D_{V,P}]$ , so that  ${f}_{V,P_0},\dots, {f}_{V,P_k}\in I_V$ .  Since   $[D_{V,P}]\leq \beta-(k+1)\epsilon_2\beta$, then  $({f}_{V,P_0},\dots, {f}_{V,P_k}:F)$ coincides with $({f}_{V,P_0},\dots, {f}_{V,P_k},F_{k+1},\dots, F_{2k+1}:F)$ for $\alpha\in \operatorname{Cl}(\Pj^{2k+1}_{\Sigma})$ less than or equal to $\beta-(k+1)([D_{V,P}])$. Then we take $F_i= {f}_{V,P_i}$ for $0\leq i\leq k$ and by  Corollary \ref{biregular}, $(I_V)^{\alpha}=I^{\alpha}$. 
\end{proof}

\bigskip
\section{The tangent spaces to the Noether-Lefschetz loci}

By \cite{MorihikoSaito, WangZaffran2009}, $\Pj^{2k+1}_{\Sigma}$ has a pure Hodge structure,  there is a well defined residue map for it, and we can use it to construct the tangent space at a point of the Noether-Lefschetz locus. This is again basically done as in \cite{Otwinowska2003}, however we provide more details, and use the 
properties of the residue map as developed in \cite{BatyrevCox} for simplicial toric varieties.

Let $X=\{f=0\}$ be a quasi-smooth hypersurface in $\Pj_{\Sigma}$, with $\deg f =\beta$. Denote by $i:X\hookrightarrow \Pj_{\Sigma}$  the inclusion, and by $i^*:H^{\bullet}(\Pj_{\Sigma}^{2k+1},\Q)\rightarrow H^{\bullet}(X,\Q)$ the associated morphism in cohomology; $i^*:H^{2k}(\Pj_{\Sigma}^{2k+1},\Q)\rightarrow H^{2k}(X,\Q)  $ is injective by Lefschetz's hyperplane theorem. 

\begin{dfn} The primitive cohomology group $H^{2k}_{\mbox{\rm\footnotesize prim}}(X,\Q)$ is the quotient $$H^{2k}(X,\Q)/i^*(H^{2k}(\Pj_{\Sigma}^{2k+1}, \Q)$$

\end{dfn}
Both $H^{2k}(\Pj_{\Sigma}^{2k+1}, \Q)$ and $H^{2k}(X,\Q)$ have pure Hodge structures, and the morphism $i^*$ is compatible with them, so that $H^{2k}_{\mbox{\rm\footnotesize prim}}(X)$ inherits a pure Hodge structure.

We  denote by $M$ the dual lattice of the lattice $N$ which contains the fan $\Sigma$, i.e., $\Sigma\subset N\otimes  \R$.  

\begin{dfn}Fix an integer basis  $m_1,\dots m_{2k+1}$ for the lattice $M$. Then given a subset $\iota=\{i_1,\dots, i_{2k+1}\}\subset \{1,\dots,\#\Sigma(1) \} $, where $\#\Sigma(1)$ is the number of rays in $\Sigma$, we define
$$\det(e_{\iota}):=\det\big(<m_j,e_{i_{h}}>_{1\leq j,h\leq 2k+1} \big) ;$$
moreover, $dx_{\iota}=dx_{i_1}\wedge\cdots \wedge dx_{i_{2k+1}} $ and $\hat{x}_{\iota}=\Pi_{i\notin\iota}x_{\iota}$.
\end{dfn}

\begin{dfn} The $(2k+1)-$form $\Omega_0\in \Omega^{2k+1}_S$ is defined as 
$$\Omega_0:=\sum_{|\iota|=2k+1}\det(e_{\iota})\hat{x}_{\iota}dx_{\iota}$$
where the sum is over all subsets $\iota\subset \{1,\dots ,2k+1\}$ with $2k+1$ elements. 
\label{euler}
\end{dfn}
For more details about these definitions  see \cite{BatyrevCox}. We denote by $NL_{\lambda,U}$ the local Noether-Lefschetz locus  as in Definition \ref{localNL}.

\begin{thm}$T_{[f]}\big(NL_{\lambda,U}\big)\cong E^{\beta}$, where 
$$E=\Bigl\{K\in S^{\bullet} \mid  \sum_{i=1}^b \lambda_i \int_{\operatorname{Tub}\gamma_i}  \frac{KR\Omega_0}{f^{k+1}}=0 \;for\;all\; R\in S^{N-\bullet}\Bigr\} ,$$
 $N=(k+1)\beta-\beta_0$ and $\operatorname{Tub}(-)$ is the adjoint of the residue map.
\end{thm}

\begin{proof}

By \cite[Prop.~2.10]{BruzzoGrassi2012} the $p$-th residue map 
$$r_p: H^0(\Pj_{\Sigma},\Omega^{2k+1}_{\Pj_{\Sigma}}(2k+1-p)X)\rightarrow H^{p,2k-p}_{\mbox{\rm\footnotesize prim}}(X) \;\;for\;\; 0 \leq p \leq 2k$$
is well defined; it is  surjective and has  kernel 
$H^0(\Pj_{\Sigma}, \Omega^{2k+1}_{\Pj_{\Sigma}}(2k-p)X)+dH^0(\Pj_{\Sigma},\Omega^{2k}_{\Pj_{\Sigma}}(2k-p)X).$
So
$$\res H^0(\Omega^{2k+1}(2k+1)X)=r_{2k}H^0(\Omega^{2k+1}(X))\oplus \cdots\oplus r_0H^0(\Omega^{2k+1}(2k+1)X))$$ 
by definition of $H^0(\Omega^{2k+1}(2k+1)X)$. Or, equivalently,
$$\res  H^0(\Omega^{2k+1}(2k+1)X)=H^{2k,0}_{\mbox{\rm\footnotesize prim}}(X)\oplus \dots \oplus H^{0,2k}_{\mbox{\rm\footnotesize prim}}(X)=H^{2k}_{\mbox{\rm\footnotesize prim}}(X). $$
Similarly
$$ \res  H^0(\Omega^{2k+1}(kX)=F^{k+1}H^{2k}_{\mbox{\rm\footnotesize prim}}(X).$$
On the other hand by \cite[Thm~9.7]{BatyrevCox}  we have  
$$H^0(\Omega_{\Pj_{\Sigma}}^{2k+1}(kX)=\left\{\frac{K\Omega_0}{f^{k}}\mid K\in S^{k\beta-\beta_0}\right\}.$$

Now fixing a basis $\{\gamma_i\}_{i=1}^b $ for $H_{2k}(X,\Q)$ we have that the coordinates of any element in $F^{k+1}H^{2k}_{\mbox{\rm\footnotesize prim}}(X)$ are  

$$\left(\int_{\gamma_1}\res \frac{K\Omega_0}{f^{k}},\dots ,\int_{\gamma_{b}}\res \frac{K\Omega_0}{f^{k}}\right),$$ 
or, equivalently,
$$\left(\int_{\operatorname{Tub}(\gamma_1)}\frac{K\Omega_0}{f^{k}},\dots ,\int_{\operatorname{Tub}(\gamma_{b})}\frac{K\Omega_0}{f^{k}}\right)$$
where $\operatorname{Tub}(\gamma_j)$ is the adjoint to the residue map. 
Now taking  $0\neq \lambda_f \in H^{k,k}(X,\Q)$ one has $\lambda_f\perp F^{k+1}H^{2k}_{\mbox{\rm\footnotesize prim}}(X)$ (see \cite{Voisin1989}) and since the sheaf $\mathcal{H}^{2k}$ is constant on $U$
we have   
$$NL_{\lambda,U}=\{g\in U\mid \lambda_g\in F^{k}H^{2k}_{\mbox{\rm\footnotesize prim}}(X_g) \}= \{g\in U\mid \lambda_f \perp F^{k+1}H^{2k}_{\mbox{\rm\footnotesize prim}}(X_g)\}.$$
More explicitly, if $(\lambda_1,\dots, \lambda_b)$ are the components of $\lambda_f $, one gets 
$$
\lambda_f \perp F^{k+1}H_{\mbox{\rm\footnotesize prim}}^{2k}(X_g) \Leftrightarrow    \sum_{i=1}^b\lambda_i\int_{\operatorname{Tub}\gamma_i} \frac{K\Omega_0}{g^{k}}=0 \; \forall K\in S^{N-\beta}
$$
where $N$ is equal to $(k+1)\beta-\beta_0$. Thus  we can characterize the local Noether-Lefschetz locus in the following way.

Let us consider the differentiable map $\psi$ which assigns to every homogeneous polynomial $g\in S^{\beta}$ a linear map $\Lambda_g\in (S^{N-\beta})^{\vee}$, i.e., 
$\Lambda: S^{\beta} \longrightarrow (S^{N-\beta})^{\vee}$
sends $g$ to 
\begin{eqnarray}
\Lambda_g\colon\  S^{N-\beta} & \to  & \C  \\ K&\mapsto&
\sum_{i}\lambda_i\int_{\operatorname{Tub}(\gamma_i)}\frac{K\Omega_0}{g^{k}};
\end{eqnarray}
then $NL_{\lambda,U}^{}=\psi_{|U}^{-1}(0)$ , hence the tangent space at $f$ is the kernel of $d\psi_{f}$. Now $T_{[f]}U\simeq S^{\beta}$. Thus we can identified canonically $T_{[f]}\big(NL_{\lambda,U}\big)$ with the subspace $E^{\beta}\subset S^{\beta}$, which is the $\beta$-summand of the Cox-Gorenstein ideal
$$E=\{K\in S^{\bullet} \mid \forall R\in S^{N-\bullet},\; \sum_{i=1}^b \lambda_i \int_{\operatorname{Tub}\gamma_i}  \frac{KR\Omega_0}{f^{k+1}}=0 \} $$
whose socle degree is $N=(k+1)\beta-\beta_0$.
\end{proof}

\begin{rmk} Note that $E$ contains the Jacobian ideal $J(f)$, and hence the toric Jacobian ideal $J_0(f)$,
since the tangent space at the point $f$ can be identified with a subspace of $(S/J(f))^{\beta}$ by Proposition 13.7 in \cite{BatyrevCox}. Moreover, by Proposition \ref{ToricJacobian} $J_0(f)$ is generated by $(2k+2)$-polynomials of the same degree $\beta$ when $f$ is general. 
\end{rmk}

We also have the Cox-Gorenstein ideals 
$$E_s:=\{ K\in S^{\bullet} \mid \forall R\in S^{N+s\beta-\bullet},\; \sum_{i=1}^b \lambda_i \int_{\operatorname{Tub}\gamma_i}  \frac{KR\,\Omega_0}{f^{k+s+1}}=0 \}$$
with $s\in \N^{+}$, which have socle degree $N+s\beta$.
For a fixed $s$, the ideal $E_s$ describes the deformation of order $s+1$ of $NL_{\lambda,U}$ in a neighborhood of $f$. Although we need only consider  $E_1$, we prove the following properties for every $s$.

\begin{prop}\label{properties} The Cox-Gorenstein ideals $E_s$ have the following properties:
\begin{enumerate}[label=\roman*)]\itemsep=-2pt
\item $E_s=(E_{s+1}:f)$;
\item If $f$ is a generic point of ${NL}_{\lambda,U}^{red}$, then $(E_s)^2\Theta\subset E_{s+1},$ where $\Theta\subset S^{\beta}$ is the image of the tangent space $T_{f}({NL}_{\lambda,U})^{red}$;
\item $\forall K\in E_s$ and $\forall j\in \{1,\dots, r \}$, $\frac{\partial K}{\partial x_j}f-(k+s+1)K\frac{\partial f}{\partial x_j}\in E_{s+1}$.
\end{enumerate}
\end{prop}

\begin{proof} 
1. Clear.

2. For every $g\in {NL}_{\lambda,U}$ and for every $\alpha\in \operatorname{Cl}(\Pj^{2k+1}_{\Sigma})$ such that $N+r\beta-\alpha$ is effective, consider the bilinear map 
$$
\begin{array}{cccl}
\mathcal{Q}_{\alpha} (g): & S^{\alpha}\times S^{N+s\beta-\alpha} & \rightarrow & \C \\
&(K,R)&\mapsto & \sum_{i=1}^b\lambda_i \int_{\operatorname{Tub}\gamma_i}\frac{KR\Omega_0}{g^{k+r+1}}
\end{array}
$$
For a fixed $R$, we have    $\ker\mathcal{Q}_{\alpha}(g)=E_s^{\alpha}(g)$, and  for a fixed $K$ we have  $\ker\mathcal{Q}_{\alpha}(g)=E_s(g)^{N+s\beta-\alpha}$, where $E_s(g)$ is the Cox-Gorenstein ideal associated to the class $\lambda_g$. Since $f$ is a quasi-smooth point of $({NL}_{\lambda, U})^{red}$, the map $g\mapsto \mathcal{Q}_{\alpha}(g)$ has constant rank for every $g$ close to $f$. So for each $\vec{v}\in T_{f}({N}_{\lambda, U})^{red}$ associated to $M\in \Theta$ the differential of the bilinear map
$$
\begin{array}{cccl}
d\mathcal{Q}_{\alpha}(f)(\vec{v}): & S^{\alpha}\times S^{N+s\beta-\alpha}& \rightarrow & \C \\
&(K,R)&\mapsto & -(k+s+2)\sum_{i=1}^t\lambda_i \int_{\operatorname{Tub}\gamma_i}\frac{KRM\Omega_0}{f^{k+s+2}}
\end{array}
$$
 is zero on $E^{\alpha}_s\times E^{N+s\beta-\alpha}_s$,  or, in other words,  $E^{\alpha}_sE^{N+s\beta-\alpha}_s\Theta\subset E^{N+(s+1)\beta}_{s+1}$. 

3. Given $K\in E_s$, for every $R\in S^{N+s\beta+\deg(x_i)-\deg(K)}$ we have
$$R\left(\frac{\partial K}{\partial x_i}f-(k+s+1)K\frac{\partial f}{\partial x_i} \right)=\underbrace{\frac{\partial(KR)}{\partial x_i}f-(k+s+1)KR\frac{\partial f}{\partial x_i}}_A-\underbrace{KF\frac{\partial R}{\partial x_i}}_B .$$
Since  $\frac{A\Omega_0}{f^{k+r+2}}$ is an exact form,  then  $A\in E_{s+1}$. By assumption $K\frac{\partial R}{\partial x_j}\in E_s$ so $B\in E_{s+1}$ by the first property.  Thus $R(\frac{\partial K}{\partial x_i}f-(k+s+1)K\frac{\partial f}{\partial x_i})\in E_{s+1}$ and since $R$ is arbitrary and $E_{s+1}$ is Cox-Gorenstein, we get the result.
\end{proof}

\section{Proof of the main theorem}

We start with a useful Lemma; the techniques used in its proof will also be used in the proof of the main result. Let $V$ be as in Corollary \ref{FundamentalCor}, i.e., $I_V=(F_0,\dots, F_k:F)$ where each $F_i$  defines  a cone with base $V$ and  vertex $P_i$ such that $\deg (F_i)\leq \upsilon$, where $\deg_{\phi}V=\phi(\upsilon)$ . 
\begin{lma}  Let $\upsilon\in \operatorname{Cl}(\Pj^{2k+1}_{\Sigma})$ be a class  such that $\phi(\upsilon)=\deg_{\phi} V$;  then for every  $L\in \left( I_V^{\leq \upsilon} \right)^2$ such that $L\in E$ one has  $L\in E_1$.
\end{lma}

\begin{proof}By assumption $L\in E$ and by Proposition \ref{properties} (ii),  $\Theta\subset (E_1:L)$; by assumption ${\rm codim} \Theta \leq d\frac{n^k}{k!}$ and hence ${\rm codim} (E_1:L)^{\beta}  \leq d\frac{n^k}{k!}$ so if $L\notin E_1$, the ideal $(E_1:L)$ satisfies the assumptions of Lemma \ref{FundamentalLemma} for $\epsilon_2=\frac{1}{2(k+1)}$ and $\delta_2=\delta$. Hence  we have,

$$l_i(E_1:L)\leq \frac{n}{2(k+1)} \ \ \text{for all} \ \ i\in\{0,\dots, k\};$$
on the other hand, by Proposition \ref{properties} (iii)  by taking $K\in \C$  one has $x_i\frac{\partial f}{\partial x_i}\in E_1$ and thus
$$l_i(E_1:L)\leq n \ \ \text{for all} \ \ i\in\{k+1,\dots, 2k+1\}.$$
Now, by assumption 
$$\phi( N+\beta -\deg L) \geq \phi(N+\beta)-2\deg_{\phi}V \geq \phi(N)+ (1-4\delta)n, $$
contradicting the fact that $E_1$ contains    a   complete intersection ideal whose generators have  ample classes,  and thus  
$$(k+1)\frac{n}{2(k+1)}+(k+1)n\leq \phi(N)+\frac{n}{2}$$
so taking $\delta<\frac{1}{8}$ we get a contradiction
\end{proof}

 Eventually, we have all the machinery and strategy necessary to prove our main result. 

\begin{thm}\label{MainThm} For every  $\epsilon>0$ there exists $\delta>0$ such that for all $n\geq \frac{1}{\delta}$ and  for all $ d\in [1,n\delta]$, if:
\begin{itemize}
    \item $\sum_{i=0}^k \deg (F_i)$ is a regular class for $(F_0,\dots, F_{2k+1})$ where $F_{k+1},\dots F_{2k+1}$  are $k$-generators of $J_0(f)$; 
    \item   for each  irreducible subvariety $Z$ of $V$, if  $Z'$ is a  subvariety $Z'\subset V$ of smallest dimension  such that $I_V=I_Z\cap I_{Z'}$ where $\phi(\kappa)=\deg_{\phi}(Z)$ and $\phi(\kappa')=\deg_{\phi}(Z')$,  then $\kappa+2\kappa'$ is a biregular class for $(F_0,\dots, F_{2k+1})$;
    \item $\codim NL_{\lambda,U}\leq d\frac{n^k}{k!}$ where $n=\phi(\beta)$,
    
\end{itemize}
then  $f\in NL_{\lambda,U}$,  $X_f$ contains  a $k$-dimensional subvariety whose $\phi$-degree is less than or equal to  $(1+\epsilon)d$.
\end{thm}

\begin{proof} It is enough to prove that $f\in I_V$ and since $V$ has pure dimension $k$ it is enough to prove $f\in \sqrt{I_Z}$ for every irreducible subvariety $Z$ of $V$. Let $Z'$ be the smallest subscheme of $V$ such that $I_V=I_Z\cap I_{Z'}$ and let  $P$ be a linear subvariety disjoint from $Z$ and $Z'$, and  let $D_{Z,P}$ and $D_{Z',P}$  be their corresponding cones with base $Z$ and vertex $P$ and base $Z'$ and vertex $P$ respectively (see Subsection \ref{phidegree}), where $P$ is the same vertex defining the cones associated with $V$. Thus $\phi([D_{Z,P}])\leq \deg_{\phi}(Z)$ and $\phi([D_{Z,P}])\leq \deg_{\phi}(Z)$. Let $K_Z$ and $K_{Z'}$ be the polynomials associated to $D_{Z,P}$ and $D_{Z,P}$, of  degree $\kappa$ and $\kappa'$ respectively. Now,   set $K_P=K_{Z,P}. K_{Z',P}^2$; by the previous Lemma $K\in E$ and since no partial derivative of $K_{Z,P}$  belong to $I_Z$, then $Z$ is generically smooth,    if   $x_i\frac{\partial K}{\partial x_i}\ne 0$ then it does not belong to $I_V$ and since $I_V$ coincides with  $E$ in   degree
$ \kappa+2\kappa'$, the latter 
degree is biregular, so that we have   $K\notin E_1$ by (i) and (iii) in Proposition \ref{properties}. Hence $(E_1:K)$ is a Cox-Gorenstein ideal with socle degree $N+\beta-(\kappa+2\kappa')$.  

Now, since $\kappa+2\kappa'$ is biregular and less than or equal to  $2 \upsilon$ (by construction), then     $I_W^{\kappa}K_P\subset  E$, and   $I_{W}^{\kappa}\subset (E_1:K_P)^{\kappa}$ by the previous Lemma. Hence $(E_1:K_P)$ contains the ideal 
$$J_P:=(f, I_W^{\kappa},x_{k+2}\frac{\partial f}{\partial x_{k+2}},\dots, x_{2k+1}\frac{\partial f}{\partial x_{2k+1}}) $$
satisfying:

\begin{itemize}
    \item $l_i(J_P)\leq \phi(\kappa) $ for $i\in \{0,\dots, k\}$;
    \item $l_i(J_P)\leq n-1 $ for $i\in \{k+1,\dots, 2k\}$;
    \item $l_i(J_P)\leq n$.
\end{itemize}
 Hence $J_P$ contains a complete intersection ideal $I'$
with ample generators,  such that  $\Lambda'\in S^{N'}$ where $N'= (k+1)[D_{Z,P}]+(k+1)\beta-\beta_0\leq N+(k+1)[D_{V,P}]$ and thus
$$\phi(N')\leq \phi(N)+(k+1)\deg_{\phi}V;  $$
on the other hand since $J_P\subset (E_1:K_P)$ and taking $\delta<\frac{1}{2(k+3)}$ and by the fact that $(E_1:K_P)$ is a Cox-Gorenstein ideal with socle degree $ N+\beta-(\kappa+2\kappa') \geq N+\beta-2[D_{V,P}] $, one has 
$$\phi(N)+n-\phi(\kappa+2\kappa')\geq \phi(N)+n-2\deg_{\phi}V;$$
this  implies   that $\delta \geq \frac{1}{2(r-(k+1)+2)}\geq \frac{1}{2(k+3)}$, which contradicts our choice of $\delta$. Thus $f\in I_Z$.
\end{proof}

\begin{cor}  For every positive 
$\epsilon$ there is  positive $ \delta$ such that for every $n\geq \frac{1}{\delta}$ and $d\in [1,n\delta]$,  if $\Pj^{2k+1}_{\Sigma}$ has Picard rank 1 and $\codim NL^{}_{\lambda,U}\leq d\frac{n^k}{k!}$ where $\phi(\beta)=n$,  then  every element of $NL^{}_{\lambda,U}$ contains a $k$-dimensional subvariety whose $\phi$-degree is less than or equal to $(1+\epsilon)d$.
    
\end{cor}

 \end{document}